\documentclass{article}
\usepackage[utf8]{inputenc}

\usepackage{amsmath}
\usepackage{lipsum}
\usepackage{amsfonts}
\usepackage{graphicx}
\usepackage{epstopdf}
\usepackage{algorithm}
\usepackage{algorithmic}
\usepackage[caption=false]{subfig}
\usepackage{hyperref}
\RequirePackage[amsmath,thmmarks,hyperref]{ntheorem}[1.33]
\theoremstyle{plain}
\newtheorem{lemma}{Lemma}

\newtheorem{definition}{Definition}
\newtheorem{theorem}{Theorem}
\usepackage{enumitem}
\setenumerate[1]{itemsep=1pt,partopsep=0pt,parsep=\parskip,topsep=3pt}
\setitemize[1]{itemsep=1pt,partopsep=0pt,parsep=\parskip,topsep=3pt}

\hypersetup{hidelinks}

\usepackage[verbose=true,letterpaper]{geometry}
\AtBeginDocument{
  \newgeometry{
    textheight=9in,
    textwidth=6.5in,
    top=1in,
    headheight=14pt,
    headsep=25pt,
    footskip=30pt
  }
}
\usepackage{indentfirst}

\newcommand{\newremark}[2]{
  \theoremstyle{plain}
  \theoremheaderfont{\normalfont\itshape}
  \theorembodyfont{\normalfont}
  \theoremseparator{}
  \theoremsymbol{}
  \newtheorem{#1}[theorem]{#2}
}

\newremark{remark}{Remark}
\newremark{assumption}{Assumption}

\newcommand{\newproblem}[2]{
  \theoremstyle{nonumberplain}
  \theoremheaderfont{\normalfont\bfseries}
  \theorembodyfont{\normalfont}
  \theoremseparator{}
  \theoremsymbol{}
  \newtheorem{#1}[theorem]{#2}
}
\newproblem{problem}{Problem}

\theoremstyle{nonumberplain}
\theoremheaderfont{\normalfont\bfseries}
\theorembodyfont{\normalfont}
\theoremseparator{.}
\theoremsymbol{\vbox{\hrule height0.6pt\hbox{\vrule height1.3ex width0.6pt\hskip0.8ex\vrule width0.6pt}\hrule height0.6pt}}
\newtheorem{proof}{Proof}
\newcommand\blfootnote[1]{%
    \begingroup  
    \renewcommand\thefootnote{}\footnote{#1}
    \addtocounter{footnote}{0}
    \endgroup
}

\title{Bayesian Nash Equilibrium Seeking for Distributed Incomplete-information Aggregative Games\blfootnote{This work was supported by the National Natural Science Foundation
of China (No. 62173250).}}
\author{Hanzheng Zhang\footnote{Key Laboratory of Systems and Control, Academy of Mathematics and Systems Science, Beijing, China (zhanghanzheng@amss.ac.cn, qin@iss.ac.cn).}\ \footnote{School of Mathematical Sciences, University of Chinese Academy of Sciences, Beijing, China.}, Guanpu Chen\footnote{JD Explore Academy, Beijing, China (chengp@amss.ac.cn).}, Huashu Qin\footnotemark[2]}
\date{}

\newcommand\keywords[1]{\textbf{Keywords}: #1}

\begin{document}

\maketitle

\begin{abstract}
In this paper, we consider a distributed Bayesian Nash equilibrium (BNE) seeking problem in incomplete-information aggregative games, which is a generalization of Bayesian games and deterministic aggregative games. We handle the aggregation function for distributed incomplete-information situations. Since the feasible strategies are infinite-dimensional functions and lie in a non-compact set, the continuity of types brings barriers to seeking equilibria. To this end, we discretize the continuous types and then prove that the equilibrium of the derived discretized model is an $\epsilon$-BNE. On this basis, we propose a distributed algorithm for an $\epsilon$-BNE and further prove its convergence.\\
\par\keywords{aggregative games, Bayesian games, equilibrium approximation, distributed algorithms}
\end{abstract}

\section{Introduction}
In recent years, distributed design for multi-agent decision and control has become increasingly important and many distribution algorithms have been proposed for various games \cite{chen2021,koshal2016,liang2017,xu2022}. Aggregative games, as non-cooperative distributed games, are widely investigated. In aggregative games, each player's cost function depends on its action and an aggregate of the decisions taken by all players, which is obtained via network communication. {\cite{koshal2016} proposed distributed synchronous and asynchronous algorithms for aggregative games, and analyzed their convergence, while \cite{liang2017} considered coupled constraints in aggregative games and provided a distributed continuous-time algorithm for the generalized Nash equilibrium. In addition, \cite{xu2022} proposed a distributed approximation algorithm using inscribed polyhedrons to estimate local set constraints.}

Considering uncertainties in reality, there are various incomplete-information models, and among them, Bayesian games are one of the most important and have a wide range of application \cite{akkarajitsakul2011,chen2020,grobhans2013,krishna2009}. In Bayesian games, players cannot obtain complete characteristics of the other players, which are called types subjected to a distribution. Each player knows its own type and has access to the distribution of all types \cite{harsanyi1967}. Due to the broad applications, the existence and computation of the Bayesian Nash equilibrium (BNE) are fundamental problems in the investigation of various Bayesian games. To this end, many works have studied the BNE of discrete-type games \cite{akkarajitsakul2011,bhaskar2016}, by fixing the types and converting the games to deterministic ones. In addition to the centralized models, there are also many works on distributed Bayesian games \cite{akkarajitsakul2011,krishnamurthy2013}, where players make decisions based on their local and neighbor's information.

However, most of the aforementioned works focus on discrete-type Bayesian games. In fact, continuous-type Bayesian games are also widespread in various fields such as engineering and economics \cite{grobhans2013,krishna2009}. The continuity of types poses challenges in seeking and verifying BNE. Specifically, in these games, the feasible strategies are infinite-dimensional functions and thus their sets are not compact \cite{guos2021,milgrom1985}. Lack of compactness, we cannot apply the fixed point theorem for the existence of BNE, let alone seek a BNE. Fortunately, many pioneers have tried to study the existence of BNE in such continuous-type situations and design its computation. For instance, \cite{milgrom1985} analyzed the existence of BNE in virtue of equicontinuous payoffs and absolutely continuous information, while \cite{meirowitz2003} investigated the situation when best responses are equicontinuous. Afterwards, \cite{guos2021} provided an equivalent condition of the equicontinuity and proposed an approximation algorithm. Moreover, \cite{ui2016} regarded the BNE as the solution to the variational inequality and provided a sufficient condition of the existence of BNE, while \cite{guow2021} gave two variational-inequality-based algorithms when the forms of strategies are prior knowledge.

Therefore, using a Bayesian scheme to analyze an incomplete-information aggregative game is worth investigating, because it can be regarded as a generalization of both deterministic aggregative games \cite{koshal2016,liang2017,xu2022} and Bayesian games \cite{guos2021,harsanyi1967,milgrom1985}. Nevertheless, continuous-type Bayesian aggregative games are more challenging than deterministic aggregative games and discrete-type Bayesian games. {On the one hand, in the incomplete-information models, since the strategies are functions of random variables, i.e., types, as the aggregate of strategies, the aggregation function should also be a function of a random variable, while the existing aggregation functions for deterministic cases \cite{koshal2016,liang2017,xu2022} cannot be applied to the incomplete-information cases.} On the other hand, to seek a continuous-type BNE in a distributed manner, we need an effective method to convert the infinite-dimensional BNE seeking problem into a finite-dimensional one, which also has to be friendly to distributed design.

Specifically, we consider seeking a continuous-type BNE in distributed aggregative games in this paper, where each player has its own type following a joint distribution, and makes decisions based on its type, local information, and the aggregate of all players' decisions. Players exchange their information and estimate the aggregate via time-varying graphs. The challenges lie in how to handle the aggregation function with incomplete-information and how to seek a BNE in this continuous-type model. The contributions are summarized as follows.

\begin{itemize}
\item We consider a distributed aggregative Bayesian game with continuous types, where each player has access to its own type and the aggregate. Such generalized models can be regarded as not only multi-player Bayesian games \cite{guos2021,harsanyi1967,milgrom1985} if each player has access to strategies of all players, but also deterministic aggregative games \cite{koshal2016,liang2017,xu2022} by letting out the uncertainties. Moreover, we focus on the incomplete-information aggregation function when players adopt non-single-valued functions as strategies, which can turn to the average of strategies when types are deterministic \cite{koshal2016}.

\item We provide a BNE approximation method by discretizing the continuous types. By establishing a discretized model, we prove that the BNE of the derived model is an $\epsilon$-BNE of the continuous-type model. Compared with existing methods \cite{guos2021,huangl2021} on continuous-type Bayesian games, our method provides an explicit error bound as well as a practical implementation beyond heuristics \cite{guow2021}.

\item Based on the discretization, we propose a gradient-descent based distributed algorithm for seeking a BNE of the discretized model, namely an $\epsilon$-BNE of the original continuous-type model. Furthermore, we prove that the proposed algorithm generates a sequence convergent to an $\epsilon$-BNE of the original continuous-type model using Lyapunov theory.
\end{itemize}

The paper is arranged as follows. Section \ref{sec_pre} summarizes the preliminaries. Section \ref{sec_form} formulates the problem. Section \ref{sec_dis} provides a discretization method to generate an $\epsilon$-BNE, while Section \ref{sec_alg} gives a distributed algorithm for the derived approximate BNE and analyzes the convergence of the algorithm. Section \ref{sec_exp} provides numerical simulations for illustration. Finally, Section \ref{sec_con} concludes the paper.

\section{Preliminaries}\label{sec_pre}
\subsection{Notations}
Denote the $n$-dimensional real Euclidean space by $\mathbb{R}^n$ and its measure by $\mu$. $B(a,\varepsilon)$ is a ball with the center $a$ and the radius $\epsilon>0$. Denote $col(x_1,\dots,x_n)=(x_1^T,\dots,x_n^T)^T$ and $1_n\in\mathbb{R}$ as the column vector with all entries equal to 1. For an integer $n>0$, denote $[n]=\{1,\dots,n\}$. For column vectors $x,y\in\mathbb{R}^n$, $\left<x,y\right>$ denotes the inner product, and $\lVert\cdot\rVert$ denotes the 2-norm. For a matrix $W\in\mathbb{R}^{n\times n}$, denote its element in the $i$-th row and $j$-th column by $[W]_{ij}$, $i,j\in[n]$. A function is piecewise continuous if it is continuous except at finite points in its domain. For $\boldsymbol{x}=(x_1,\dots,x_n)$, define the vector with entries of $\boldsymbol{x}$ except for $i$ as $x_{-i}=(x_1,\dots,x_{i-1},x_{i+1},\dots,x_n)$. For $\boldsymbol{s}=x_1\times\cdots\times x_n$, denote its surface differential by $d\boldsymbol{s}$.

\subsection{Convex analysis}
A set $C\subseteq\mathbb{R}^n$ is convex if $\lambda z_1+(1-\lambda)z_2\in C,\ \forall z_1,z_2\in C$ and $0\leq\lambda\leq1$. For a closed convex set $C\subseteq\mathbb{R}^n$, a projection map $\Pi_C:\mathbb{R}^n\to C$ is defined as $\Pi_C(x)=\arg\min_{y\in C}\lVert x-y\rVert$, and holds $\left<x-\Pi_C(x),\Pi_C(x)-y\right>\geq0$, $\forall y\in C$. A function $f:\mathbb{R}^n\to\mathbb{R}$ is (strictly) convex if $f(\lambda x_1+(1-\lambda)x_2)(<)\leq \lambda f(x_1)+(1-\lambda)f(x_2)$, $\forall x_1,x_2\in\mathbb{R}^n$ and $\lambda\in(0,1)$.

For a convex differentiable function $f$, the gradient of $f$ at point $x$ is denoted by $\nabla f$, satisfying $f(y)\geq f(x)+\left<y-x,\nabla f(x)\right>$, $\forall y\in\mathbb{R}^n$. For a convex differentiable function $f(x_1,\dots,x_n)$, denote $\nabla_i f$ as the differential of $f$ with respect to $x_i$. If $f$ is (strictly) convex, the gradient of $f$ satisfies $\left<\nabla f(x)-\nabla f(y),x-y\right>(>)\geq 0$.

\subsection{Bayesian games}
Consider a Bayesian game denoted by $G=(I,\{X_i\}_{i\in I},\boldsymbol\Theta,P(\cdot),\{f_i\}_{i\in I})$ with a set of players $I=[n]$, where player $i$ has the feasible action set $X_i\subseteq\mathbb{R}^{m_i}$ and the cost function $f_i(x_i,x_{-i},\theta_i)$. For $i\in I$, the incomplete information of player $i$ is referred to the \textit{type}, denoted by $\theta_i\in\Theta_i\subseteq\mathbb{R}$, and $\boldsymbol\theta=(\theta_1,\dots,\theta_n)\in\boldsymbol\Theta$ is a random variable mapping from the probability space $(\Omega,\mathcal{B},P)$ to $\mathbb{R}^n$. Denote the density function of $P$ by $p$ with the marginal density $p_i(\theta_i)=\int_{\Theta_{-i}}p(\theta_i,\theta_{-i})d\theta_{-i}$ and the conditional probability density $p_i(\theta_{-i}|\theta_i)=p(\theta_i,\theta_{-i})/p_i(\theta_i)$, $i\in[n]$ Throughout the paper, we use $\boldsymbol\theta$ to denote a random variable mapping from $(\Omega,\mathcal{B},P)$ to $\mathbb{R}^n$, or a deterministic element in $\mathbb{R}^n$ depending on the context.

Here each player $i\in I$ only knows its own type but not those of its rivals. As in Bayesian games \cite{harsanyi1967}, the joint distribution $P$ is public information. The cost function of player $i$ is defined as $f_i:\boldsymbol X\times\Theta_i\to\mathbb{R}$, depending on all player's actions and the type of $i$. Each player adopts a strategy $\sigma_i$, which is a measurable function mapping from its type set $\Theta_i$ to its action set $X_i$, and $\sigma_i(\theta_i)$ is the action taken by player $i$ when it receives the type $\theta_i\in\Theta_i$. Denote the possible strategy set of player $i$ by $\Sigma_i$. Define Hilbert spaces $\mathcal{H}_i$ consisting of functions $\beta:\mathbb{R}\to\mathbb{R}^{m}$ with the inner product $\left<\sigma,\sigma'\right>_{\mathcal{H}_i}=\int_{\theta_i\in\Theta_i}\left<\sigma_i,\sigma_i'\right>p_i(\theta_i)d\theta_i,\ \sigma,\sigma'\in\mathcal{H}_i,i\in I.$ Thus, the strategy set $\Sigma_i$ is a subset of the Hilbert space $\mathcal{H}_i$.

\subsection{Graph theory}
An undirected graph $\mathcal{G}$ is defined by $\mathcal{G}(\mathcal{V},\mathcal{E})$, where $\mathcal{V}=[n]$ is the node set and $\mathcal{E}\subseteq\mathcal{V}\times\mathcal{V}$ is the edge set. Node $j$ is a neighbor of $i$ if $(j,i)\in\mathcal{E}$, and thus node $i$ is a neighbor of $j$. Take $(i,i)\in\mathcal{E}$. A path in $\mathcal{G}$ from $i_1$ to $i_k$ is an alternating sequence $i_1e_1i_2\cdots i_{k-1}e_{k-1}i_k$ of nodes such that $e_j=(i_j,i_{j+1})\in\mathcal{E}$ for $j\in[k-1]$. $W=([W]_{ij})\in\mathbb{R}^{n\times n}$ is the adjacency matrix such that $[W]_{ij}>0$ if $(j,i)\in\mathcal{E}$ and $[W]_{ij}=0$ otherwise. $\mathcal{G}$ is connected if there is a path in $\mathcal{G}$ from $i$ to $j$ for any pair nodes $i,j\in\mathcal{G}$.

\section{Problem Formulation}\label{sec_form}
Consider an incomplete-information aggregative game, denoted by $(I,\{X_i\}_{i\in I},\boldsymbol{\Theta},P,$
$\{f_i\}_{i\in I})$. Each player $i\in I=[n]$ has its type $\theta_i\in\Theta_i\subseteq\mathbb{R}$, feasible action set $X_i\subseteq\mathbb{R}^m$, and cost function $f_i(x_i,\tilde{x},\theta_i)$, where the type $\boldsymbol{\theta}=(\theta_1,\dots,\theta_n)$ follows the distribution $P(\boldsymbol{\theta})$ with the density $p(\boldsymbol{\theta})$ and $\tilde{x}$ is the aggregate of all players' decisions. The type set $\Theta_i$ is compact and without loss of generalization, take $\Theta_i=[\underline{\theta},\overline{\theta}]$. Player $i$ adopts a strategy $\sigma_i$, which is a measurable function from its type set $\Theta_i$ to its action set $X_i$. That is, at type $\theta_i\in\Theta_i$, player $i$ will take $\sigma_i(\theta_i)\in X_i$ as its action. Denote the strategy set of player $i$ by $\Sigma_i$. 

Different from deterministic games, consider the aggregation functions in incomplete-information situations that players take non-single-valued functions of types as strategies, which means that the aggregate ought to be a function of types. Here, the aggregation function is shown as follows.
\begin{equation}\label{eq_aggre}
\bar{\sigma}(\tilde{\theta})=\int_{\sum_{i=1}^{n}\theta_i=n\tilde{\theta}}\frac{\sum_{i=1}^n\sigma_i(\theta_i)}{n}\bar{p}(\theta_1,\dots,\theta_n|\tilde{\theta})d\boldsymbol{s},
\end{equation}
where $\boldsymbol{s}=\theta_1\times\cdots\times\theta_n$, $\tilde{\theta}=(\theta_1+\dots+\theta_n)/n$ is a random variable following the distribution
$$\bar{P}(\tilde{\theta})=\int_{\sum_{i=1}^{n}\theta_i=n\tilde{\theta}}p(\theta_1,\dots,\theta_n)d\boldsymbol{s},$$
with the density function $\bar{p}(\tilde{\theta})$ and the conditional probability density function $$\bar{p}(\tilde{\theta}|\theta_i)=\int_{\sum_{j\ne i}\theta_j=n\tilde{\theta}-\theta_i}p(\theta_i,\theta_{-i})/p_i(\theta_i)ds_{-i}.$$ Note that $\bar{\sigma}$ is a linear function with respect to $\sigma_1,\dots,\sigma_n$, denoted by $H(\sigma_1,\dots,\sigma_n)$.

{
\begin{remark}
The aggregation function \eqref{eq_aggre} can be regarded as the average of the players' strategies and $\tilde{\theta}$ is the average type. When the model is deterministic, which means that $|\boldsymbol\Theta|=\{\boldsymbol\theta_0\}$ and $P(\boldsymbol\theta_0)=1$, the aggregation function \eqref{eq_aggre} turns to the simple average of players' strategies, as \cite{koshal2016,xu2022}.
\end{remark}}

With the above aggregation function, the goal of player $i$ is to minimize the following conditional expectation of $f_i$ 
\begin{equation*}\begin{aligned}
U_i(\sigma_i,\bar{\sigma},\theta_i)=&\int_{\underline{\theta}}^{\overline{\theta}}f_i(\sigma_i(\theta_i),\bar{\sigma}(\tilde{\theta}),\theta_i)\bar{p}(\tilde{\theta}|\theta_i)d\tilde{\theta}\\=&\int_{\Theta_{-i}}f_i(\sigma_i(\theta_i),H(\sigma_i(\theta_i),\sigma_{-i}(\theta_{-i})),\theta_i)p_i(\theta_{-i}|\theta_i)d\theta_{-i}.
\end{aligned}\end{equation*}
Consequently, we can also regard $U_i(\sigma_i,\bar{\sigma},\theta_i)$ as $J_i(\sigma_i,\sigma_{-i},\theta_i)$. Denote its gradient by
$$F_i(\sigma_i,\bar{\sigma},\theta_i)=\nabla_i J_i(\sigma_i,\sigma_{-i},\theta_i).$$
Then we give the concept of Bayesian Nash equilibrium.
\begin{definition}
A strategy profile $(\sigma_1^*,\dots,\sigma_n^*)$ is a Bayesian Nash equilibrium (BNE) if for any $\sigma_i\in\Sigma_i$, $i\in I$,
$$U_i(\sigma_i^*,\sigma_{-i}^*,\theta_i)\leq U_i(\sigma_i,\sigma_{-i}^*,\theta_i),\ \mathrm{for\ a.e.\ }\theta_i\in\Theta_i.$$
\end{definition}

We make the following assumptions for the aggregative game $G$.
\begin{assumption}\label{ass_game}
Consider the incomplete-information aggregative game $G$. For $i\in I$,
\begin{enumerate}\renewcommand\labelenumi{(\roman{enumi})}
\item the action set $X_i$ is nonempty, convex, and compact;
\item the distribution $P$ is atomless, i.e., $P(\boldsymbol{\theta}=\boldsymbol{\zeta})=0$ for any given $\boldsymbol{\zeta}\in\boldsymbol{\Theta}$. Moreover, the measure $\mu(\{\theta_i\in\Theta_i|p_i(\theta_i)>0\})=\mu(\Theta_i)$;
\item the cost function $f_i(x_i,\tilde{x},\theta_i)$ is strictly convex in $x_i\in X_i$ and $L_\theta$-Lipschitz continuous in $\theta_i\in\Theta_i$ for each $x_i,\tilde{x}\in\mathbb{R}^m$;
\item the expectation $U_i$ is well defined for every $\sigma_j\in\Sigma_j$ and $\theta_i\in\Theta_i$, $j\in I$, and its gradient $F_i$ is $D$-Lipschitz continuous in $x_i\in X_i$ for any $\sigma_{-i}\in\Sigma_{-i}$ and $\theta_i\in\Theta_i$, and is $L_u$-Lipschitz continuous in $\bar{\sigma}$ for any $\sigma_i\in\Sigma_i$ and $\theta_i\in\Theta_i$.
\end{enumerate}
\end{assumption}
Assumption \ref{ass_game} was widely used in the study of aggregative games and Bayesian games \cite{guos2021,guow2021,koshal2016,milgrom1985}. The atomless property in Assumption \ref{ass_game}(ii) is a common assumption in Bayesian games \cite{guos2021,meirowitz2003,milgrom1985}, and the measure condition can be guaranteed by removing types in $\{\theta_i\in\Theta_i,\exists\varepsilon>0,p_i(\theta_i')>0,\forall\theta_i'\in B(\theta_i,\varepsilon)\}$.

Players in our model have local interactions with each other over time to estimate the aggregate, where these interactions are modeled by time-varying graphs $\mathcal{G}(t)$. At time $t$, players exchange their estimations of the aggregate with current neighbors through $\mathcal{G}(t)$, which satisfies the following assumption.

\begin{assumption}\label{ass_graph}
The graph sequence $\mathcal{G}(t)$ is uniformly jointly strongly connected, i.e., there exists an integer $\mathcal{B}>0$ such that $\cup_{k=t}^{t+\mathcal{B}} \mathcal{G}(k)$ is strongly connected, and its adjacency matrix $W(t)$ satisfies $[W(t)]_{ij}>\eta$, $\eta>0$, and $\sum_{i=1}^n [W(t)]_{ij}=\sum_{j=1}^n [W(t)]_{ij}=1$.
\end{assumption}
 
Assumption \ref{ass_graph} holds for a variety of networks and ensures the connectivity, which was also used in \cite{koshal2016}. With Assumption \ref{ass_graph}, we have the following result \cite{koshal2016,nedic2010}.

\begin{lemma}\label{lem_consensus}
Denote the transition matrices $\Phi(k,s)$ from time $s$ to $k>s$ as $\Phi(k,s)=W(k)W(k-1)\cdots W(s)$ for $0\leq s<k$. Under Assumption \ref{ass_game}(v),
\begin{enumerate}\renewcommand\labelenumi{(\alph{enumi})}
\item $\lim_{k\to\infty}\Phi(k,s)=\frac{1}{n}1 1^T$ for all $s\geq0$.
\item $|[\Phi(k,s)]_{ij}-1/n|\leq \Gamma \beta^{k-s}$ for all $k\geq s\geq0$ and $i,j\in I$, where $\Gamma=(1-\eta/(4N^2))^{1/\mathcal{B}}$ and $\beta=(1-\eta/(4N^2))^{1/\mathcal{B}}$.
\end{enumerate}
\end{lemma}

The existence of the BNE can be guaranteed by the variational inequalities \cite{guow2021,ui2016}, summarized as follows.
\begin{lemma}
Under Assumption \ref{ass_game}, there exists a unique BNE of game $G$.
\end{lemma}

Based on the existence, our goal is to compute the BNE of the proposed model, summarized as follows.
\begin{problem}
Seek the BNE of the incomplete-information aggregative game $G=(I,\{X_i\}_{i\in I},\boldsymbol{\Theta},P,\{f_i\}_{i\in I})$ in a distributed manner.
\end{problem}

In Bayesian games, the continuity of types poses barriers to seeking a BNE. As the Riesz's Lemma shows \cite{rynne2007}, any infinite-dimensional normed space contains a sequence of unit vectors $\{x_n\}$ with $\lVert x_n-x_m\rVert>\alpha$ for any $0<\alpha<1$ and $n\ne m$. Then the strategy set $\boldsymbol\Sigma$ lying in the infinite-dimensional space $\mathcal{H}_1\times\cdots\times\mathcal{H}_n$, is not compact, which poses obstacles in computation. There are a few attempts to seek a continuous-type BNE. For example, \cite{guow2021} considered the situation that the strategy forms are prior knowledge, in which the forms are usually unavailable, while \cite{guos2021} utilized polynomial approximations to estimate a BNE without the estimation error. Moreover, \cite{huangl2021} adopted heuristic approximations in discrete-action Bayesian games, but their method was NP-hard and not practical to be implemented in continuous-action games. Thus, since directly seeking a BNE is hard, we introduce the following concept.
\begin{definition}
Denote $EU(\sigma_i,\bar{\sigma})=\int_{\Theta_i}U_i(\sigma_i,\bar{\sigma},\theta_i)p_i(\theta_i)d\theta_i.$ For any $\epsilon>0$, a strategy profile $\hat{\boldsymbol{\sigma}}^*=(\hat{\sigma}_1^*,\dots,\hat{\sigma}_n^*)$ is an $\epsilon$-Bayesian Nash equilibrium ($\epsilon$-BNE) of $G$ if for any $\sigma_{i}\in\Sigma_{i}$, $i\in I$,
$$EU(\sigma_i,H(\sigma_i,\hat{\sigma}_{-i}^*)\geq EU(\hat{\sigma}_i^*,H(\hat{\boldsymbol\sigma}^*)).$$
\end{definition}

To overcome the above bottlenecks in seeking BNE, we propose a discretization method in the following section to convert the infinite-dimensional problem to a finite-dimensional one.

\section{Discretization}\label{sec_dis}

In this section, we give a discretization method and show its effectiveness in approximating the best responses and BNE of the continuous-type model $G$.

For each player $i$, we select $N$ points $\theta_i^k$ from $\Theta_i$ satisfying
$$P_i(\theta_i^k)=\frac{k}{N},\ k\in [N].$$
Denote the corresponding discrete type set by $\hat{\Theta}_i$. Define $\theta_i^0=\underline{\theta}$ and $\hat{\boldsymbol{\Theta}}=\hat{\Theta}_1\times\cdots\times\hat{\Theta}_n$. In the discretized model, we regard all types in the interval $(\theta_i^{k-1},\theta_i^k]$ as $\theta_i^k$, then the discrete type $\boldsymbol{\theta}$ follow the below joint distribution
\begin{equation}\label{eq_dispoints}
\hat{P}(\boldsymbol{\theta})=\frac{1}{N^n},\ \boldsymbol{\theta}\in\hat{\boldsymbol{\Theta}}.
\end{equation}
Correspondingly, the marginal distribution $\hat{P}_i(\theta_i)=\sum_{\theta_i\in\hat{\Theta}_i}\hat{P}(\theta_i,\theta_{-i})$ and the conditional distribution $\hat{P}_i(\theta_{-i}|\theta_i)=\hat{P}(\theta_i,\theta_{-i})/\hat{P}_i(\theta_i).$

Due to Assumption \ref{ass_game}(ii) that $\mu(\{\theta_i\in\Theta_i|p_i(\theta_i)>0\})=\mu(\Theta_i)$, the gap between the adjacent discrete points $\theta_i^k-\theta_i^{k-1}$ tends to 0 as $N$ tends to infinity. Since we use $\theta_i^k$ to represent the interval $(\theta_i^k,\theta_i^{k-1}]$, we choose the length of such intervals as small as possible, which can effectively reduce the error. Additionally, our design is friendly to distributed algorithms, while other choices of discrete points will bring extra computation when players update the strategies.

Based on the above discretization, we formulate a discretized model as $\hat{G}=(I,\{X_i\}_{i\in I},$
$\hat{\Theta},\hat{P},\{f_i\}_{i=1}^{n})$. In this model, strategies are restricted to $N$-dimensional vectors. Denote the strategy set of player $i$ in $\hat{G}$ by $\hat{\Sigma}_i$. Thus, the aggregate of the discretized strategies is 
\begin{equation*}
\hat{\bar{\sigma}}(\tilde{\theta})=\sum_{\sum_{i=1}^n\theta_i=n\tilde{\theta},\theta_i\in\hat{\Theta}_i}\frac{\sum_{i=1}^n\hat{\sigma}_i(\theta_i)}{n}\hat{\bar{P}}(\theta_1,\dots,\theta_n|\tilde{\theta}),
\end{equation*}
where $\tilde{\theta}=\sum_{i=1}^n\theta_i/n$ follows the below discrete distribution
$$\hat{\bar{P}}(\tilde{\theta})=\sum_{\sum_{i=1}^n\theta_i=n\tilde{\theta}, \theta_i\in\hat{\Theta}_i}\frac{1}{N^n},$$
with the conditional probability $\hat{\bar{P}}(\cdot|\tilde{\theta})$.
Since the $\hat{\bar{P}}$ is a discrete distribution, the aggregation function can be written as $\hat{\bar{\sigma}}=\sum_{i=1}^n h_i(\hat{\sigma}_i)$, where $h_i$ is a linear function mapping from $\hat{\Sigma}_i$ to $\mathbb{R}^{nNm}$. Denote $h_{-i}(\hat{\sigma}_{-i})=\sum_{j\ne i}h_j(\hat{\sigma}_j)$ for $\hat{\sigma}_{-i}\in\hat{\Sigma}_{-i}$, then the expectation of the cost of $i$ in $\hat{G}$ is 
$$\hat{U}_i(\hat{\sigma}_i,\hat{\bar{\sigma}},\theta_i)=\sum_{\theta_{-i}\in\hat{\Theta}_{-i}}f_i(\hat{\sigma}_i(\theta_i),h_i(\hat{\sigma}_i(\theta_i))+h_{-i}(\hat{\sigma}_{-i}(\theta_{-i})),\theta_i)\hat{P}_i(\theta_{-i}|\theta_i),\ \theta_i\in\hat{\Theta}_i.$$
Similar to the continuous-type model $G$, $\hat{U}_i$ can also be regarded as $\hat{J}_i(\hat{\sigma}_i,\hat{\sigma}_{-i},\theta_i)$. Denote the best response strategy of player $i$ with respect to $\hat{\sigma}_{-i}\in\hat{\Sigma}_{-i}$ in $\hat{G}$ by $\hat{\sigma}_{i*}^N$, which satisfies for any $\theta_i\in\hat{\Theta}_i$,
$$\hat{\sigma}_{i*}^N=\arg\min_{\hat{\sigma}(\theta_i)\in X_i}\hat{J}_i(\sigma_i,\hat{\sigma}_{-i},\theta_i).$$ Denote the set of best responses by $BR_i^N(\hat{\sigma}_{-i})$ in $\hat{G}$. Then we define the following best response and equilibrium of $\hat{G}$.
\begin{definition}
A strategy pair $(\hat{\sigma}_1^*,\dots,\hat{\sigma}_n^*)$ is a BNE of $\hat{G}$, or a DBNE($N$) of $\hat{G}$ if 
$$\hat{\sigma}_i^*\in BR^N_i(\hat{\sigma}_{-i}^*).$$
\end{definition}
The existence of DBNE can be guaranteed by variational inequalities \cite{guow2021,ui2016} or Browner fixed point theorem \cite{guos2021,milgrom1985}, summarized as follows.
\begin{lemma}
Under Assumption \ref{ass_game}, there exists a unique DBNE($N$) of the discretized model $\hat{G}$.
\end{lemma}

To approximate the strategies in the continuous-type model $G$ with the strategies in the discretized model $\hat{G}$, we extend the domains of strategies from $\hat{\Theta}_i$ to $\Theta_i$, and define the strategies of $\hat{G}$ at type $\theta_i\in(\theta_i^{k-1},\theta_i^k]$ as
$$\hat{\sigma}_i(\theta_i)=\hat{\sigma}_i(\theta_i^k).$$
Denote $\hat{\tilde{\Theta}}=\{\tilde{\theta}|n\tilde{\theta}=\sum_{i=1}^n\theta_i,\theta_i\in\Theta_i\}=\{\tilde{\theta}^1,\dots,\tilde{\theta}^S\}$, $S\leq nN$, and $\tilde{\theta}^0=\underline{\theta}$. For $\hat{\sigma}_i\in\Sigma_i$, $i\in I$, their aggregates in the discretized model $\hat{G}$ and the continuous-type model $G$ satisfy
$$\bar{\sigma}(\tilde{\theta})=\hat{\bar{\sigma}}(\tilde{\theta}^k),\ \tilde{\theta}\in(\tilde{\theta}^{k-1},\tilde{\theta}^k].$$ 
With this extension, we denote $\hat{\bar{\sigma}}$ of the discretized strategies as $\bar{\sigma}$ for convenience.

Then we estimate the best response in $G$. Actually, a best response $\hat{\sigma}_{i*}$ needs to respond to any strategies in $\Sigma_{-i}$, rather than strategies in $\hat{\Sigma}_{-i}\subseteq\Sigma_{-i}$. To this end, we modify the best responses $\hat{\sigma}_{i*}$ with respect to $\sigma_{-i}\in\Sigma_{-i}$ as follows. For $\theta_i^k\in\hat{\Theta}_i$,
$$\hat{\sigma}_{i*}(\theta_i^k)=\arg\min_{\hat{\sigma}_i(\theta_i)\in X_i}\int_{\Theta_{-i}} f_i(\hat{\sigma}(\theta_i^k),H(\hat{\sigma}_i(\theta_i^k),\sigma_{-i}(\theta_{-i})),\theta_i^k)\frac{\int_{\theta_i^{k-1}}^{\theta_i^k}p(\theta_i,\theta_{-i})d\theta_i}{\hat{P}_i(\theta_i^k)}d\theta_{-i}.$$

The following lemma shows the relation between the best response in the discretized model $\hat{G}$ and in the continuous-type model $G$.

\begin{lemma}\label{lem_br}
For a given $\sigma_{-i}\in\Sigma_{-i}$, if all the best responses in $BR_i^N(\sigma_{-i})$ of $G$ are piecewise continuous, then the best responses in $BR_i^N(\sigma_{-i})$ in $\hat{G}$ are almost surely the best responses of $G$, as $N$ tends to infinity. Specifically, for any $\hat{\sigma}_{i*}^N\in BR_i^N(\sigma_{-i})$, there exists $\sigma_{l*}\in BR_i(\sigma_{-i})$ such that
$$\lim_{N\to\infty}\hat{\sigma}_{i*}^N(\theta_i)=\sigma_{i*}(\theta_i),\ \mathrm{for\ a.e.\ }\theta_i\in\Theta_i.$$
\end{lemma}

\begin{proof}
Due to the L'Hospital's rule,
$$\lim_{N\to\infty}\frac{\int_{\theta_i^{r-1}}^{\theta_i^k}p(\theta_i,\theta_{-i})}{\int_{\theta_i^{r-1}}^{\theta_i^k}p_i(\theta_i)d\theta_i}=\frac{p(\theta_i^k,\theta_{-i})}{p_i(\theta_i^k)}=p_i(\theta_{-i}|\theta_i^k).$$
Thus, for any $\hat{\sigma}_{i*}\in BR_i^N(\sigma_{-i})$ and $\theta_i\in\hat{\Theta}_i$, there exists a strategy $\sigma_{i*}\in BR_i(\sigma_{-i})$ such that, as $N$ tends to infinity, $\hat{\sigma}_{i*}^N(\theta_i)\to \sigma_{i*}$. Since $\sigma_{i*}$ is piecewise continuous, for any $\varepsilon>0$, there exists $\delta>0$ such that, for any $\theta_i\in\Theta_i$ except for finite points and $\theta_i'\in B(\theta_i,\delta)\cap \Theta_i$, $|\sigma_{i*}(\theta_i)-\sigma_{i*}(\theta_i')|<\varepsilon$. Take $\varepsilon=\max_k\{\theta_i^k-\theta_i^{k-1}\}$, thus
$$\int_{\Theta_i}|\hat{\sigma}_{i*}^N(\theta_i)-\sigma_{i*}(\theta_i)|^2d\theta_i\leq\sum_{i=1}^N\varepsilon^2(\theta_i^k-\theta_i^{k-1})=\varepsilon^2(\overline{\theta}-\underline{\theta}).$$
As $N$ tends to infinity, $\varepsilon$ tends to 0 and thus, $\hat{\sigma}_{i*}^N$ is $\sigma_{i*}$ for almost every $\theta_i\in\Theta_i$.
\end{proof}

Lemma \ref{lem_br} implies that players can utilize the best responses derived from $\hat{G}$ to estimate the best responses in $G$. That is to say, players are willing to adopt the best responses in $\hat{G}$, and thus these best responses form a DBNE. Then we give the relation between the derived DBNE of $\hat{G}$ and the BNE of $G$. 

In the next theorem, we show that the DBNE($N$) of $\hat{G}$ is an $\epsilon$-BNE of $G$. In addition, we provides an explicit error bound of our approximation, compared with heuristic approximations \cite{guos2021,guow2021,huangl2021}.

\begin{theorem}\label{thm_bne}
Let $(\sigma_1^*,\dots,\sigma_n^*)$ be the BNE of the continuous-type model $G$. Under Assumption \ref{ass_game}(i), (ii), and (iv), the DBNE $(\hat{\sigma}_1^*,\dots,\hat{\sigma}_n^*)$ of the discretized model $\hat{G}$ is an $\epsilon$-BNE, where $\epsilon=O(\max_{i\in I}\max_{k\in[N]}(\theta_i^k-\theta_i^{k-1}))$.
\end{theorem}

\begin{proof}
Due to the Lipschitz continuity of $f_i$, $|f_i(x_i,x_{-i},\theta_i)-f_i(x_i,x_{-i},\theta_i')|\leq L_\theta \lVert \theta_i-\theta_i'\rVert$ for $x_i\in X_i,x_{-i}\in X_{-i},\theta_i,\theta_i'\in\Theta_i$. Denote $EU_i(\sigma_i,\bar{\sigma})=\int_{\Theta_i} f_i(\sigma_i,\bar{\sigma},\theta_i)p_i(\theta_i)d\theta_i$ and $E\hat{U}_i=\sum_{\hat{\Theta}_i}f_i(\hat{\sigma}_i,\bar{\sigma},\theta_i)\hat{P}_i(\theta_i)$. Then for any $\hat{\sigma}_i\in\hat{\Sigma}_i$ and $\hat{\sigma}_{-i}\in\hat{\sigma}_{-i}$,
\begin{equation}\label{eq_EU}
\begin{aligned}
EU_i(\hat{\sigma}_i,\bar{\sigma})=&
\sum_{k=1}^N \int_{\theta_i^{k-1}}^{\theta_i^k} U_i(\hat{\sigma}_i,\bar{\sigma},\theta_i)p_i(\theta_i)d\theta_i\\
=&\sum_{k=1}^N \int_{\theta_i^{k-1}}^{\theta_i^k} (U_i(\hat{\sigma}_i,\bar{\sigma},\theta_i)-U_i(\hat{\sigma}_i,\bar{\sigma},\theta_i^k)+U_i(\hat{\sigma}_i,\bar{\sigma},\theta_i^k))p_i(\theta_i)d\theta_i\\
\leq& E\hat{U}_i(\hat{\sigma}_i,\bar{\sigma})+L(\overline{\theta}-\underline{\theta})\epsilon_0.
\end{aligned}
\end{equation}

Denote the aggregate of the DBNE by $\hat{\bar{\sigma}}^*$. From the definition of DBNE, for any $\theta_i\in\hat{\Theta}_i$ and $\hat{\sigma}_i\in \hat{\sigma}_i$,
\begin{equation}\label{eq_DBNE}
\hat{U}_i(\hat{\sigma}_i^*,\hat{\bar{\sigma}}^*,\theta_i)\leq \hat{U}_i(\hat{\sigma}_i,h_i(\hat{\sigma}_i)+h_{-i}(\hat{\bar{\sigma}}^*_{-i}),\theta_i).
\end{equation}
Then we convert the DBNE from the discretized model to the continuous-type model. Due to Assumption \ref{ass_game}(ii), the distribution $P$ is continuous and thus $p$ is $L_p$-Lipschitz continuous over $\boldsymbol{\Theta}$. Therefore, for any $\theta_i\in(\theta_i^{k-1},\theta_i^k]$, $k\in[N]$,
\begin{equation*}\begin{aligned}
&\left|p(\theta_i,\theta_{-i})-\frac{1}{\theta_i^k-\theta_i^{k-1}}\int_{\theta_i^{k-1}}^{\theta_i^k}p(\theta_i',\theta_{-i})d\theta_i'\right|\\
\leq &\frac{1}{\theta_i^k-\theta_i^{k-1}}\int_{\theta_i^{k-1}}^{\theta_i^k}|p(\theta_i,\theta_{-i}),p(\theta_i',\theta_{-i})|d\theta_i'\leq L_p\epsilon_0.
\end{aligned}
\end{equation*}
Thus, for any $\sigma_i\in\Sigma_i$ and $k\in[N]$,
\begin{equation}\label{eq_convert}
\begin{aligned}
&\int_{\theta_i^{k-1}}^{\theta_i^k}\int_{\Theta_{-i}}f_i(\sigma_i(\theta_i),H(\sigma_i(\theta_i),\hat{\bar{\sigma}}^*_{-i}(\theta_{-i})),\theta_i^k)p(\theta_i,\theta_{-i})d\theta_id\theta_{-i}\\
\geq&\int_{\theta_i^{k-1}}^{\theta_i^k}\int_{\Theta_{-i}}f_i(\sigma_i(\theta_i),H(\sigma_i(\theta_i),\hat{\bar{\sigma}}^*_{-i}(\theta_{-i})),\theta_i^k)\frac{\int_{\theta_i^{k-1}}^{\theta_i^k}p(\theta_i',\theta_{-i})d\theta_i'}{\theta_i^k-\theta_i^{k-1}}d\theta_i d\theta_{-i}\\&-ML_p(\theta_i^k-\theta_i^{k-1})\mu(\Theta_{-i})\epsilon_0\\
=&\hat{U}_i(\sigma_i,H(\sigma_i,\hat{\bar{\sigma}}^*_{-i}),\theta_i^k)\hat{P}_i(\theta_i^k)-ML_p(\theta_i^k-\theta_i^{k-1})\mu(\Theta_{-i})\epsilon_0.
\end{aligned}
\end{equation}
With \eqref{eq_DBNE} and \eqref{eq_convert}, for any $\sigma_i\in\Sigma_i$,
\begin{equation}\label{eq_DBNE_1}\begin{aligned}
E\hat{U}_i(\hat{\sigma}_i^*,\hat{\bar{\sigma}}^*)\leq& \sum_{i=1}^N \int_{\theta_i^{k-1}}^{\theta_i^k}\int_{\Theta_{-i}}(f_i(\sigma_i(\theta_i),H(\sigma_i(\theta_i),\hat{\bar{\sigma}}^*_{-i}(\theta_{-i})),\theta_i^k)p(\theta_i,\theta_{-i})d\theta_id\theta_{-i}\\&+ML_p(\theta_i^k-\theta_i^{k-1})\mu(\Theta_{-i})\epsilon_0)\\
=&EU(\sigma_i,\hat{\bar{\sigma}}^*)+ML_p\mu(\boldsymbol{\Theta})\epsilon_0.\end{aligned}
\end{equation}
Define $C=ML_p\mu(\boldsymbol{\Theta})+L_\theta$. Combining \eqref{eq_EU} and \eqref{eq_DBNE_1}, for any $\sigma_i\in\Sigma_i$,
$$EU(\hat{\sigma}_i^*,\hat{\bar{\sigma}}^*)\leq E\hat{U}(\hat{\sigma}_i^*,\hat{\bar{\sigma}}^*)+L_\theta\epsilon_0\leq EU_i(\sigma_i,H(\sigma_i,\hat{\bar{\sigma}}^*_{-i}))+C\epsilon_0,$$
which means that the DBNE is an $\epsilon$-BNE with $\epsilon=C\epsilon_0$.
\end{proof}

\section{Distributed algorithm}\label{sec_alg}

In this section, we propose the distributed algorithm for the BNE of the discretized model $\hat{G}$, namely an $\epsilon$-BNE of the continuous-type model $G$. 

At time $t$, player $i$ estimates the aggregate according to neighbors' approximations as
\begin{equation}\label{eq_est}
u_i^t=\sum_{j=1}^n [W(t)]_{ij}v_j^t,
\end{equation}
where $v_j^t$ is the approximation of the aggregate made by player $j$ at time $t$.
Due to the uncertainties, we can also regard $u_i^t(\tilde{\theta})$ as a function mapping from $\hat{\tilde{\Theta}}$ to $\mathbb{R}^m$, where $\tilde{\theta}$ follows the discrete distribution $\hat{\bar{P}}$ as defined in Section \ref{sec_dis}. Then player $i$ evaluates its subgradient as
\begin{equation}\label{eq_gradient}
g_i^t=(F_i(\sigma_i^t,u_i^t,\theta_i^1),\dots,F_i(\sigma_i^t,u_i^t,\theta_i^N))/N,
\end{equation}
where $F_i(\sigma_i^t,u_i^t,\theta_i^k)$ ($k\in[N])$ was defined in Section \ref{sec_form}. We summarize the above procedures as follows.\\ \vspace{-1em}

\begin{algorithm}[h]
\caption{Algorithm for an $\epsilon$-BNE of continuous-type Bayesian games}\label{alg_main}\label{alg}
\begin{algorithmic}
\STATE \textbf{Initialization}: For $i\in I$: take $\sigma_{i}(0)\in\widetilde{\Sigma}_i$, and $v_i(0)=h_i(\sigma_i(0))$. 	
\STATE \textbf{Discretization}: For $i\in I$, take $N$ discrete points from the type set $\Theta_i$ as \eqref{eq_dispoints}.
\STATE Iterate until $t\geq T$:
\STATE \textbf{Communicate and Update}: Player $i$ evaluates the aggregate of neighbors $u_i(t)$ based on \eqref{eq_est} and the gradient $g_i^t$ based on \eqref{eq_gradient}, then updates $\sigma_{i}(t)$ and its observation $v_i(t)$ by
$$\sigma_i^{t+1}=\Pi_i(\sigma_i^t-\alpha(t)g_i^t),$$ 
$$v_i^{t+1}=u_i^t-h_i(\sigma_i^t)+h_i(\sigma_i^{t+1}).$$
\end{algorithmic}
\end{algorithm}

The stepsize $\alpha(t)$ taken in Algorithm \ref{alg} satisfies
\begin{enumerate}\renewcommand\labelenumi{(\alph{enumi})}
\item $\alpha(t)$ is a positive non-increasing sequence.
\item $\sum_{t=0}^\infty \alpha(t)=\infty$, $\sum_{t=0}^\infty \alpha^2(t)<\infty$.
\end{enumerate}

Therefore, we give the following main result of this paper to show the convergence of Algorithm \ref{alg} to an $\epsilon$-BNE, or the DBNE($N$) with an explicit error bound $\epsilon$.
\begin{theorem}\label{thm_conv}
Under Assumption \ref{ass_game} and \ref{ass_graph}, Algorithm \ref{alg_main} generates a sequence that convergent to the DBNE of $\hat{G}$, which is an $\epsilon$-BNE of $G$, with $\epsilon=O(\max_{i\in I}\max_{k\in[N]}(\theta_i^k-\theta_i^{k-1}))$.
\end{theorem}





\begin{proof}
Define the average of the estimations $v_i^t$ as $\bar{v}^t=\sum_{i=1}^n v_i^t/n$. Firstly, we prove that $\bar{v}^t=\sum_{i=1}^n h_i(\sigma_i^t)/n$ by induction on $t$.

For $t=0$, the above relation holds trivially. Since the adjacency matrices have columns sum up to 1,  assume that holds for $t-1$, as the induction step,
\begin{align*}
\sum_{i=1}^n v_i^t=&\sum_{i=1}^n (u_i^{t-1}+h_i(\sigma_i^t)-h_i(\sigma_i^{t-1}))=\sum_{i=1}^n\bigg(\sum_{j=1}^n[W(t-1)]_{ij}v_i^{t-1}+h_i(\sigma_i^t)-h_i(\sigma_i^{t-1})\bigg)\\
=&\sum_{i=1}^n (v_i^{t-1}+h_i(\sigma_i^t)-h_i(\sigma_i^{t-1}))=\sum_{i=1}^n h_i(\sigma_i^t).
\end{align*}
Thus, $\bar{v}^t=\sum_{i=1}^n h_i(\sigma_i^t)/n$ holds for all $t\geq0$.

Secondly, we establish a relation between the estimations $u_i^t$ and the average $\bar{v}^t$. From the update rule of $v_i^t$, for $t\geq 1$,
\begin{equation}\label{eq_est1}
u_i^{t}=\sum_{j=1}^n [\Phi(t,0)]_{ij}v_j^0+\sum_{r=1}^{t}\sum_{j=1}^n[\Phi(t,r)]_{ij}(h_i(\sigma_j^r)-h_i(\sigma_j^{r-1})).
\end{equation}
Then, for $t\geq 1$, $\bar{v}^t$ can be reconstructed as
\begin{equation}\label{eq_ave_est}
\bar{v}^t=\frac{1}{n}\sum_{j=1}^n  v_j^0+\frac{1}{n}\sum_{r=1}^{t}(\sum_{j=1}^n (h_i(\sigma_j^r)-h_i(\sigma_j^{r-1}))).
\end{equation}
Since $h_i$ is linear, there exists a constant $C>0$ such that $h_i$ is $C$-Lipschitz continuous for $i\in I$. Due to the property of projection, $\lVert\Pi_{\hat{\Sigma}_i}( \sigma_i^t-\alpha(t)g_i^t)-\sigma_i^t\rVert\leq\alpha(t)\lVert g_i^t\rVert\leq \alpha(t)D$ for $\sigma_i^t\in\hat{\Sigma}_i$. Combining \eqref{eq_est1} and \eqref{eq_ave_est}, with Lemma \ref{lem_consensus}, for $t\geq 1$,
\begin{equation*}
\lVert u_i^t-\bar{v}^t\rVert\leq \Gamma \beta^tR_0+C\Gamma D\sum_{r=1}^{t} \beta^{t-r}\alpha(r),
\end{equation*} 
where $R_0=\sum_{j=0}^n v_j^0$. Since $\alpha(t)$ is non-increasing,
\begin{equation}\label{eq_aggregate}
\sum_{t=0}^{\infty} \alpha(t)\lVert v_i^t-\bar{v}^t\rVert \leq \Gamma \frac{1}{1-\beta}R_0\alpha(0)+C\Gamma D \frac{1}{1-\beta}\sum_{t=0}^\infty \alpha^2(t).
\end{equation}
Because the stepsize $\alpha(t)$ satisfies $\sum_{t=0}^\infty \alpha(t)=\infty$ and $\sum_{t=0}^\infty \alpha^2(t)<\infty$, $\sum_{t=0}^\infty\alpha(t)\lVert v_i^t-\bar{v}^t\rVert<\infty$.

Thirdly, we give the convergence result. With the update rule of $\sigma_i^t$,
\begin{equation*}
\lVert \sigma_i^{t+1}-\hat{\sigma}_i^*\rVert=\lVert \Pi_{\hat{\Sigma}_i}(\sigma_i^t-\alpha(t)g_i^t)-\hat{\sigma}_i^*\rVert
\leq \lVert \sigma_i^t-\hat{\sigma}_i^*-\alpha(t)(g_i^t+g_i^*)\rVert
\end{equation*}
where $g_i^*=col(F_i(\hat{\sigma}_i^*,\hat{\bar{\sigma}}^*,\theta_i^1),\dots,F_i(\hat{\sigma}_i^*,\hat{\bar{\sigma}}^*,\theta_i^N))$. Then
\begin{equation}\label{eq_lya}
\lVert \sigma_i^{t+1}-\hat{\sigma}_i^*\rVert^2\leq \lVert \sigma_i^t-\hat{\sigma}_i^*\rVert^2+\alpha^2(t)\lVert g_i^t-g_i^*\rVert^2-2\alpha(t)\left<g_i^t-g_i^*,\sigma_i^t-\hat{\sigma}_i^*\right>
\end{equation}
To prove that $\sigma_i^t$ converges to $\hat{\sigma}_i^*$, we need to show that $\sum_{t=0}^\infty \alpha(t)\lVert \sigma_i^t-\sigma_i^*\rVert<\infty$. Since $\lVert g_i^t\rVert\leq \sqrt{N}D$ and $\lVert g_i^*\rVert\leq \sqrt{N}D$, 
$$\sum_{t=0}^\infty\alpha^2(t)\lVert g_i^t-g_i^*\rVert<\infty.$$
Based on the property of subgradients, for $\theta_i^k\in\hat{\Theta}_i$,
\begin{equation*}
N\left<g_i^t,\sigma_i^t-\hat{\sigma}_i^*\right>=\sum_{k=1}^N\left<F_i(\sigma_i^t,u_i^t,\theta_i^k)-F_i(\hat{\sigma}_i^*,\hat{\bar{\sigma}}^*,\theta_i^k),\sigma_i^t(\theta_i^k)-\hat{\sigma}_i^*(\theta_i^k)\right>
\end{equation*}
Since $F$ is Lipschitz continuous,
\begin{equation*}
F_i(\sigma_i^t,u_i^t,\theta_i^k)-F_i(\sigma_i^t,\bar{v}^t,\theta_i^k)\rVert\leq L_u \lVert u_i^t(\theta_i^k)-\bar{v}^t(\theta_i^k)\rVert,
\end{equation*}
where $u_i^t(\theta_i^k)$ and $\bar{v}^t(\theta_i^k)$ denoted the aggregates $u_i^t(\tilde{\theta}|\theta_i^k)$ and $\bar{v}^t(\tilde{\theta}|\theta_i^k)$ with $\tilde{\theta}\in\{\sum_{i=1}^n\theta_i/n,$
$\theta_i=\theta_i^k,\theta_{-i}\in\Theta_{-i}\}$, satisfying $\lVert u\rVert^2=\sum_{k=1}^N \lVert u(\theta_i^k)\rVert^2$ for $i\in I$ and $u=u_i^t,\bar{v}^t$. Then
\begin{equation*}\begin{aligned}
&N\left<g_i^t-g_i^*,\sigma_i^t-\hat{\sigma}_i^*\right>\\
\leq &\sum_{k=1}^N \left<F_i(\sigma_i^t,\bar{v}^t,\theta_i^k)-F_i(\hat{\sigma}_i^*,\hat{\bar{\sigma}}^*,\theta_i^k),\sigma_i^t(\theta_i^k)-\hat{\sigma}_i^*(\theta_i^k)\right>+\sqrt{N}L_u\lVert u_i^t-\bar{v}^t\rVert.
\end{aligned}
\end{equation*}
Denote $EF_i(\sigma_i,\bar{\sigma})=\sum_{k=1}^N F(\sigma_i,\bar{\sigma},\theta_i^k)$ for $\sigma\in\hat{\boldsymbol\Sigma}$, where $\bar{\sigma}$ is the aggregate of $\boldsymbol\sigma$, and $E\boldsymbol F(\boldsymbol{\sigma})=col(EF_1(\sigma_1,\bar{\sigma}),$
$\dots,EF_n(\sigma_n,\bar{\sigma}))$. From the strict convexity of $f_i$, $EU_i$ is strict convex, and thus, $$\left<E\boldsymbol{F}(\boldsymbol{\sigma}^t,H(\boldsymbol\sigma^t))-E\boldsymbol{F}(\boldsymbol{\sigma}^*,H(\hat{\boldsymbol\sigma}^*)),\boldsymbol{\sigma}^t-\hat{\boldsymbol{\sigma}}^*\right><0.$$ Combining with \eqref{eq_aggregate}, $\sum_{i=1}^n\lVert \sigma_i^t-\hat{\sigma}_i^*\rVert^2=\lVert \boldsymbol\sigma^t-\hat{\boldsymbol\sigma}^*\rVert^2$ is convergent. Furthermore, 
$$\sum_{t=0}^\infty \alpha(t)\left<E\boldsymbol{F}(\boldsymbol{\sigma}^t,H(\boldsymbol\sigma^t))-E\boldsymbol{F}(\boldsymbol{\sigma}^*,H(\hat{\boldsymbol\sigma}^*)),\boldsymbol{\sigma}^t-\hat{\boldsymbol{\sigma}}^*\right>\leq\infty.$$
Since $\sum_{t=0}^\infty\alpha(t)=\infty$, there exists a subsequence $\{t_l\}$ such that
$$\lim_{l\to\infty}\left<E\boldsymbol{F}(\boldsymbol{\sigma}^t,H(\boldsymbol\sigma^t))-E\boldsymbol{F}(\boldsymbol{\sigma}^*,H(\hat{\boldsymbol\sigma}^*)),\boldsymbol{\sigma}^t-\hat{\boldsymbol{\sigma}}^*\right>=0.$$
Due to the strict convexity of $f_i$, $E\boldsymbol F$ is strictly monotone, and thus $\lim_{l\to\infty}\lVert \boldsymbol\sigma^{t_l}-\hat{\boldsymbol\sigma}^*\rVert=0$. Since $\lVert \boldsymbol\sigma^t-\hat{\boldsymbol\sigma}^*\rVert$ is convergent,
$$\lim_{t\to\infty}\lVert \boldsymbol\sigma^t-\hat{\boldsymbol\sigma}^*\rVert=0.$$
Therefore, we complete the proof.
\end{proof}

\section{Numerical Simulations}\label{sec_exp}
In this section, we provide numerical simulations to illustrate the effectiveness of Algorithm \ref{alg} on aggregative Bayesian games.

Consider a Nash-Cournot game played by 5 competitive firms to produce a kind of commodity. Firms face uncertainties in the game, which is referred to the type $\boldsymbol\theta=(\theta_1,\dots,\theta_5)\in[0,1]^5$ and $\theta_1,\dots,\theta_5$ are independent and uniformly distributed over $[0,1]$, respectively. For firm $i$, $i\in\{1,\dots,5\}$, it has a feasible action set $X_i=[0,20]$ and cost function $f_i(x_i,\tilde{x},\theta_i)=(\tilde{x}+1200-20(i-1))x_i+\theta_i x_i^2$, where $\tilde{x}$ is the aggregate of actions. Due to the uncertainties of types, $i$ adopts a strategy $\sigma_i$, which means that when it receives a type $\theta_i\in[1,2]$, it takes $\sigma_i(\theta_i)\in X_i=[0,5]$ as the quantity of the commodity to produce. The expectation of the cost and the aggregation function is defined in Section \ref{sec_form}. Firms exchange their aggregates via a time-varying graph $\mathcal{G}(t)$, which is randomly generated.

\begin{figure}[!t]
\centering
\subfloat[$\theta_i= 1.3$]{
\includegraphics[width=6.5cm]{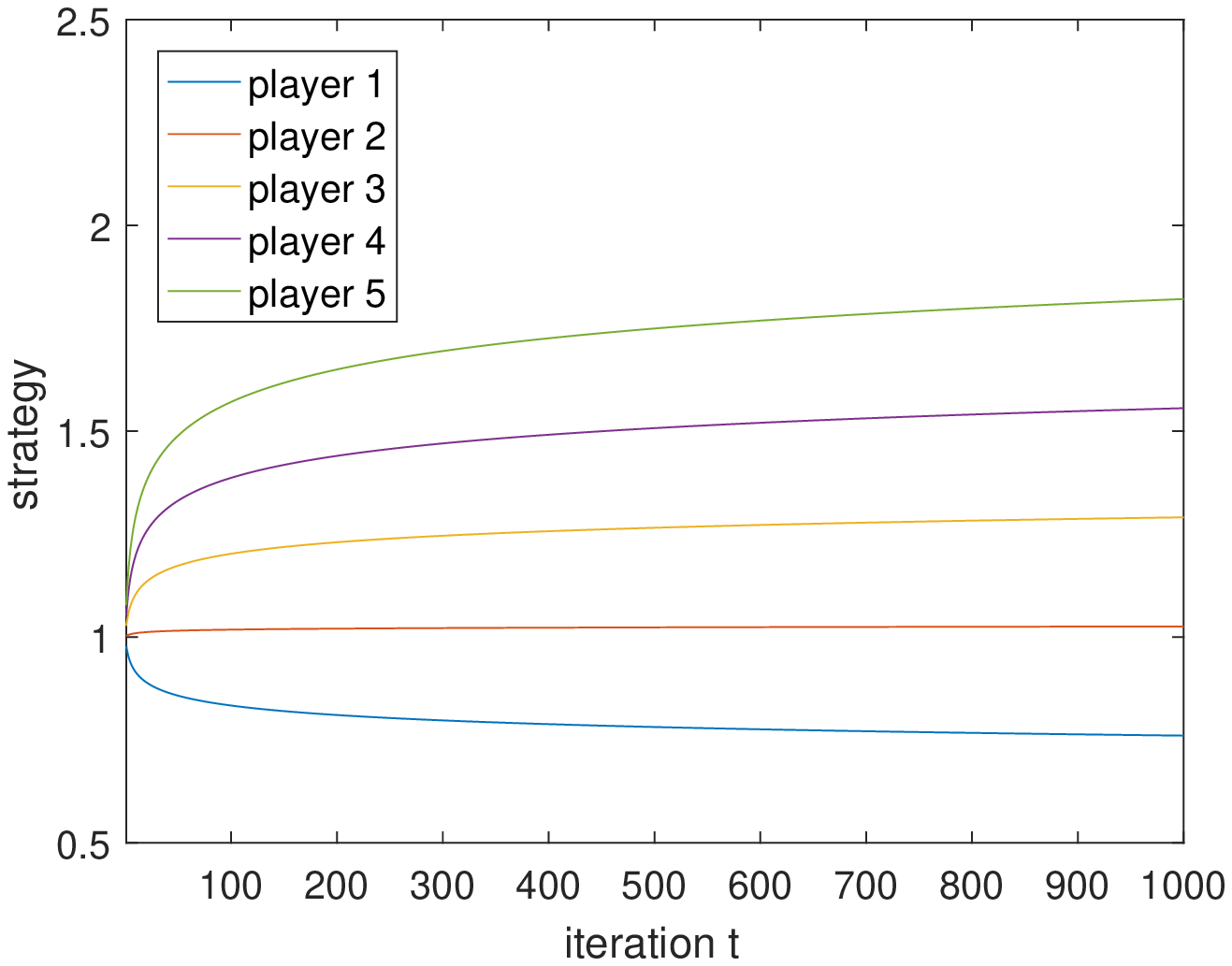}}
\hfill
\subfloat[$\theta_i= 1.7$]{\includegraphics[width=6.5cm]{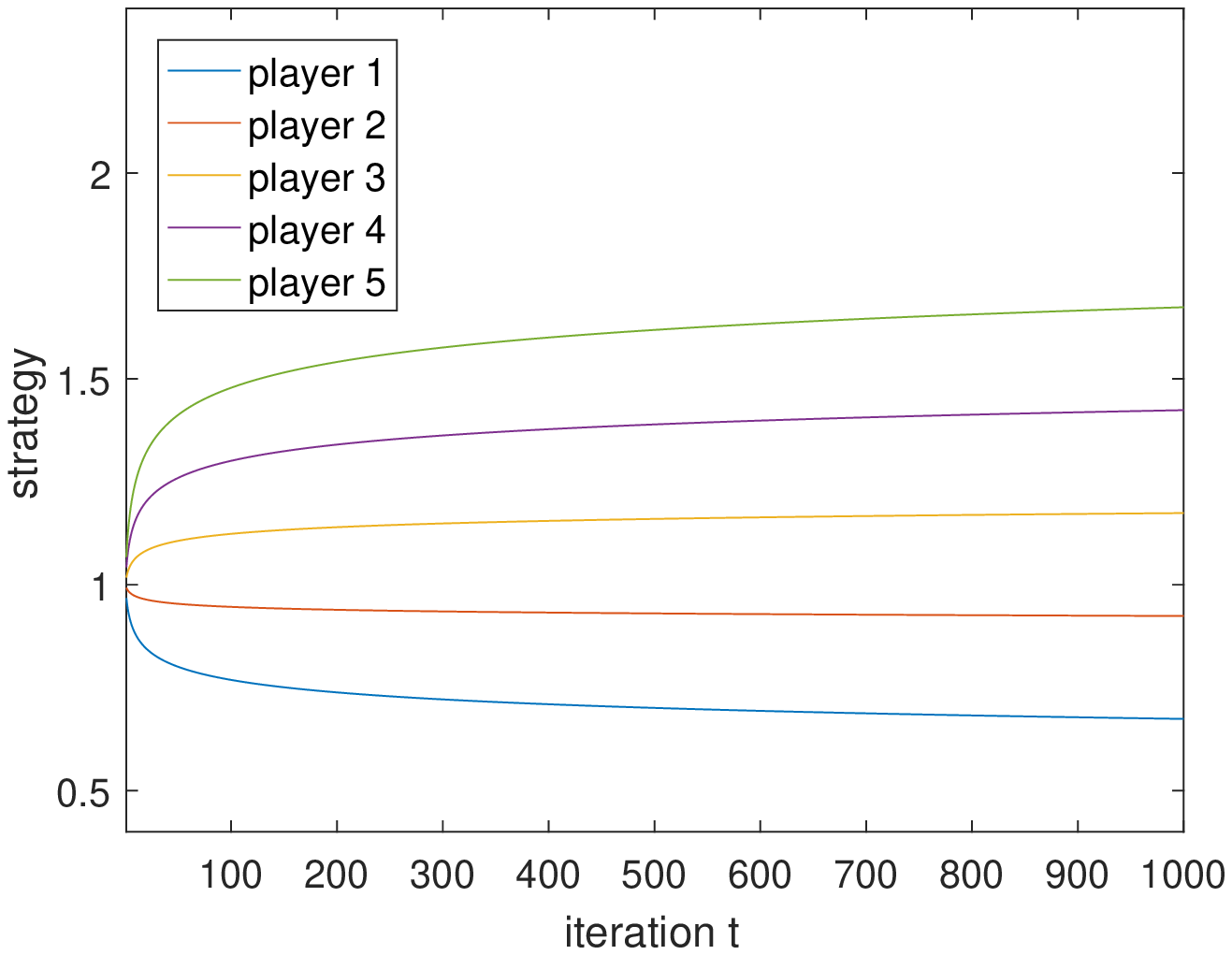}}
\caption{Strategies of players at $\theta_i= 1.3,1.7$ with $N=200$.}\label{fig_1}
\end{figure}

Firstly, we show the convergence of Algorithm \ref{alg_main}. Fig. \ref{fig_1} presents the trajectories of strategies for different players at specific type $\theta_i=1.3,1.7$ under $N=200$. We can see that the generated strategies converge.

Next, we verify the effectiveness of the discretization. Fig. \ref{fig_2} shows the trajectories of strategies for player 3 under different numbers of discrete points $N=50,100,160,200,250$. We can find that the limit points of $\sigma_3^t(1.3)$ and $ \sigma_1^t(1.7)$ converge. Since the error $\epsilon$ tends to 0 as $N$ tends to infinity, we believe that our approximations are close to the true equilibrium. 

\begin{figure}[!h]
\centering
\subfloat[$\theta_1=1.3$]{\includegraphics[width=6.5cm]{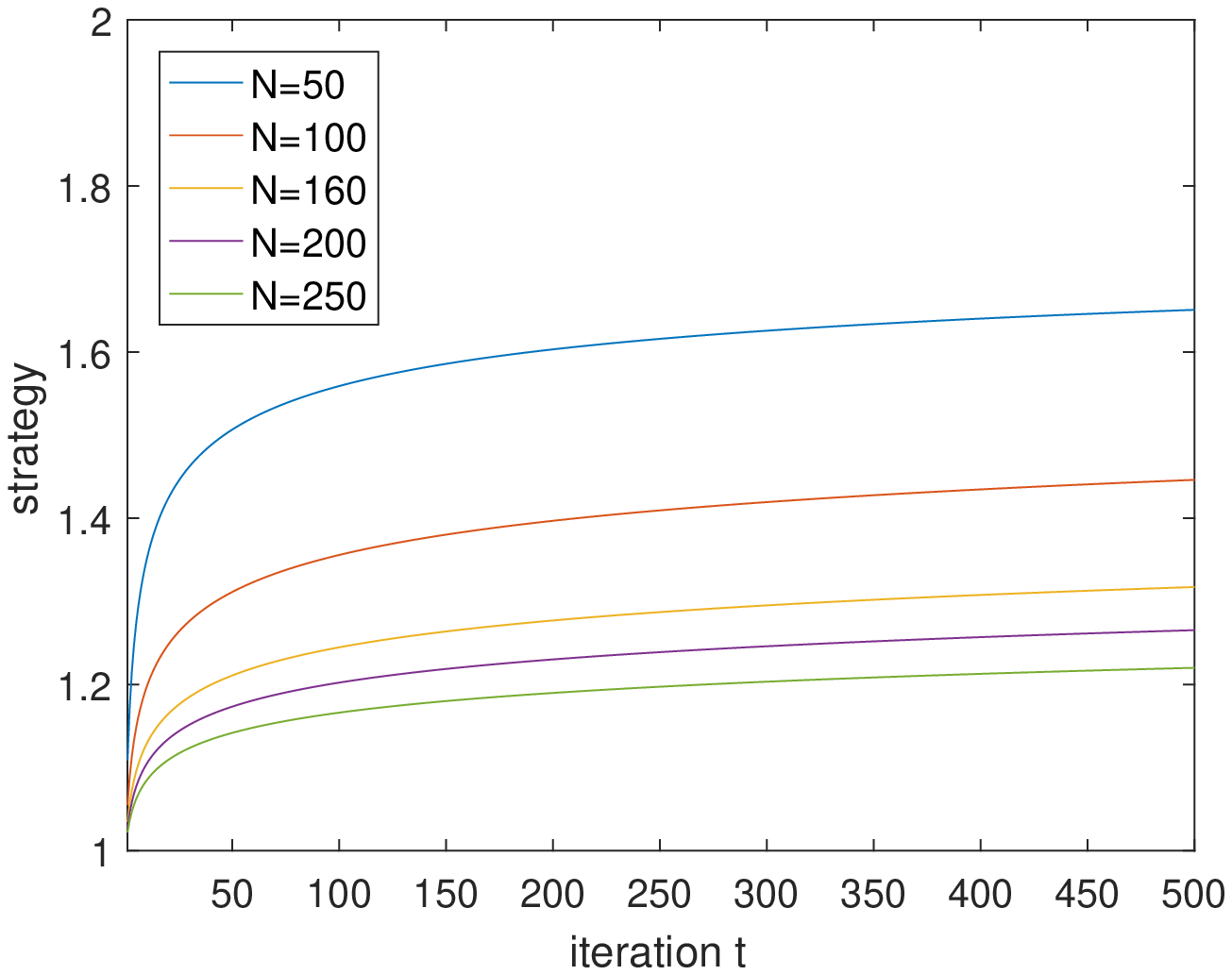}}
\hfill
\subfloat[$\theta_1=1.7$]{\includegraphics[width=6.5cm]{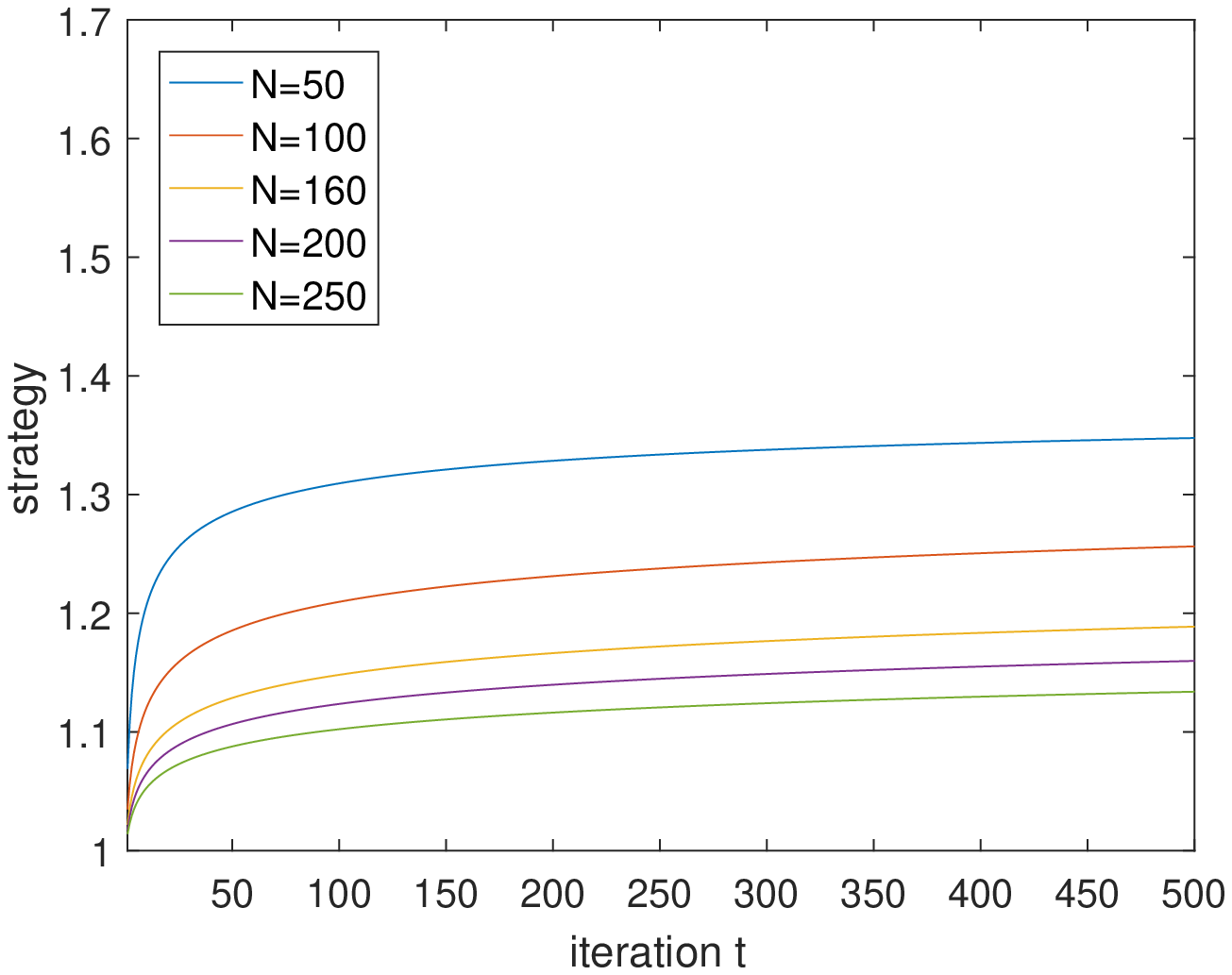}}
\caption{Strategies of player 3 at $\theta_1= 1.3,1.7$ with different $N$.}\label{fig_2}
\end{figure}

%
%

\section{Conclusion}\label{sec_con}
In this paper, we considered an aggregative Bayesian game, which is a generalization of multi-player Bayesian games and deterministic aggregative games. We handled the aggregation function with the incomplete-information situations. To break through barriers in seeking BNE, we provided a discretization method, and proved that the DBNE of the generated discretized model is an $\epsilon$-BNE of the continuous-type model with the explicit error bound. On this basis, we proposed a distributed algorithm for the DBNE of the discretized model, namely an $\epsilon$-BNE of the continuous-type model, and proved its convergence to an $\epsilon$-BNE of the continuous-type model.

\end{document}